\documentclass[10pt,a4paper]{amsart}
\usepackage{graphicx}
\usepackage{amssymb,amsmath,amsfonts}
\usepackage[colorlinks=true,linkcolor=blue,citecolor=blue,urlcolor=black,bookmarksopen=true,pagebackref=false]{hyperref}
\usepackage{bookmark}
\usepackage[top=3cm,right=2.5cm,bottom=2.88cm,left=2.3cm]{geometry}
\newtheorem{thm}{Theorem}[section]
\newtheorem{cor}[thm]{Corollary}
\newtheorem{lem}[thm]{Lemma}
\newtheorem{prop}[thm]{Proposition}
\theoremstyle{definition}

\theoremstyle{remark}
\newtheorem{rem}[thm]{Remark}
\newtheorem{exam}[thm]{Example}
\newtheorem{quest}{Question}
\numberwithin{equation}{section}
\newcommand{\norm}[1]{\left\Vert#1\right\Vert}
\newcommand{\abs}[1]{\left\vert#1\right\vert}

\begin{document}

\title{Johnson pseudo-contractibility and pseudo-amenability of $ \theta $-Lau product of Banach algebras }

\author[M. Askari-Sayah]{M. Askari-Sayah}

\email{mehdi17@aut.ac.ir}

\author[A. Pourabbas]{A. Pourabbas}
\email{arpabbas@aut.ac.ir}

\address{Faculty of Mathematics and Computer Science,
Amirkabir University of Technology, 424 Hafez Avenue, 15914 Tehran,
 Iran.}

\author[A. Sahami]{A. Sahami}
\email{amir.sahami@aut.ac.ir}

\address{Department of Mathematics, Faculty of Basic Sciences, Ilam University, P.O. Box 69315-516,
Ilam, Iran.}



% ----------------------------------------------------------------
\begin{abstract}
Given Banach algebras $ A $ and $ B $ with $ \theta\in\Delta(B) $. We shall study the Johnson pseudo-contractibility and pseudo-amenability of $ \theta $-Lau product $ A\times_{\theta} B $.  We show that if $ A\times_{\theta} B $ is Johnson pseudo-contractible, then $ A $ is Johnson pseudo-contractible and has a bounded approximate identity and $ B $ is Johnson pseudo-contractible. In some particular cases complete characterization of Johnson pseudo-contractibility of $ A\times_{\theta} B $ are given. Also, we show that pseudo-amenability of $ A\times_{\theta} B $ implies approximate amenability of $ A $ and pseudo-amenability of $ B $.
\end{abstract}
\maketitle
%--------------------%---------------------%-------------------------%--------------------------------------------------------------------------------------
\section{Introduction}
Let $ A $ and $ B $ be two Banach algebras with $ \theta\in\Delta(B) $, the character space of $ B $. Then the Banach space $ A\times B $ with the product
\begin{equation*}
  (a, b)(c, d)=(ac+\theta(d)a+\theta(b)c, bd)\qquad a, c\in A,\hspace{2mm}  b, d\in B,
\end{equation*}
and $ \ell^{1} $-norm becomes a Banach algebra, which is called the $ \theta $-Lau product of $ A $ and $ B $ and is denoted by $ A\times_{\theta} B $. The $ \theta $-Lau product was introduce first by A. T . Lau \cite{lau} for special Banach algebras, more precisely for $ F $-algebras. Recently, this product was extended by M. Monfared \cite{monfared} for any Banach algebras $ A $ and $ B $ and any character $ \theta\in\Delta(B) $. Monfared studied many properties of $ A\times_{\theta} B $ including semisimpility, Arens regularity, amenability, etc. We recall that the concept of an amenable Banach algebra was introduced by Johnson in 1972. Indeed, a Banach algebra $ A $ is called amenable if there is an element $ M\in (A\otimes_{p}A)^{**} $ such that $a\cdot M=M\cdot a$ and $\pi^{**}_{A}(M)a=a$ for every $a\in A$, where $\pi : A\otimes_{p} A \longrightarrow A$ is the  product morphism and $A\otimes_{p}A$ is denoted for the projective tensor product of $A$ with itself. Motivated by this construction of Johnson, some authors introduce several modifications
of this notion by relaxing some conditions in different versions of  definitions of amenability.
In the Section \ref{sec-joh} we deal with one of these, namely Johnson pseudo-contractibility that was introduced recently by second and third authors \cite{sah1}. A Banach algebra $ A $ is called Johnson pseudo-contractible if there is a not necessarily bounded net $(M_{\alpha}) \subseteq (A\otimes_{p}A)^{**}$ such that $a\cdot M_{\alpha}=M_{\alpha}\cdot a$ and $\pi^{**}_{A}(M_{\alpha})a-a\rightarrow 0$ for every $a\in A$.  We show that if $ A\times_{\theta} B $ is Johnson pseudo-contractible, then $ A $ is Johnson pseudo-contractible and has a bounded approximate identity and $ B $ is Johnson pseudo-contractible. Moreover, we show that in particular cases, for example when $ A $ is Arens regular and weakly sequentially complete or when $ A $ is a dual Banach algebra, Johnson pseudo-contractibility of $ A\times_{\theta} B $ is equivalent with amenability of $ A $ and Johnson pseudo-contractibility of $ B $. Some example are given at the end of the section.

 In the Section \ref{sec-pse} we shall focus on pseudo-amenability of $ A\times_{\theta} B $. This notion was introduced by  F. Ghahrami and Y. Zhang  in \cite{ghah-pse}. A Banach algebra $ A $ is called pseudo-amenable if there is a net $ (m_{\alpha})\subseteq A\otimes_{p}A $ such that $ a\cdot m_{\alpha}-m_{\alpha}\cdot a\rightarrow 0 $ and $ \pi_{A}(m_{\alpha})a\rightarrow a $ for each $  a\in A $. Pseudo-amenability of $ A\times_{\theta} B $ was studied by  E. Ghaderi {\it et al.} in \cite{ghaderi}. They prove that pseudo-amenability of $ A\times_{\theta} B $ implies pseudo-amenability of $ B $, and implies pseudo-amenability of $ A $ whenever $ A $ has a bounded approximate identity. We extend this result and we  show that the existence of bounded approximate identity is not necessary. Indeed, we show that if $ A\times_{\theta} B $ is pseudo-amenable, then $ A $ is approximately amenable and $ B $ is pseudo-amenable. The concept of approximately amenable Banach algebra was introduced by  F. Ghahrami and R. J. Loy in \cite{generalize1}, also see \cite{generalize2}. A Banach algebra $ A $ is called approximately amenable if there are nets
$ (M_{\alpha})\subseteq A\otimes_{p}A $, $ (F_{\alpha})\subseteq A $ and $ (G_{\alpha})\subseteq A $ such that for any $ a\in A $
  \begin{enumerate}
    \item[(i)] $a\cdot M_{\alpha}-M_{\alpha}\cdot a+F_{\alpha}\otimes a-a\otimes G_{\alpha}\rightarrow 0$,
    \item[(ii)] $aF_{\alpha}\rightarrow a$, \quad $ G_{\alpha}a\rightarrow a  $, \text{and}
    \item[(iii)]  $ \pi_{A}(M_{\alpha})a-F_{\alpha}a-G_{\alpha}a\rightarrow 0 $.
  \end{enumerate}
%%%%%%%%%%%%%%%%%%%%%%%%%%%%%%%%%%%%%%%%%%%%%%%%%%%%%%%%
\section{Johnson pseudo-contractibility of $ A\times_{\theta} B $}\label{sec-joh}
 We state a result from \cite{mehdi} that will be used frequently in this section.
 \begin{thm}\label{joh imply amenable}
  Let $ A $ be a Johnson pseudo-contractible Banach algebra with an
identity. Then $ A $ is amenable.
\end{thm}
\begin{lem}\label{lem  I  is Johnson pseudo-contractible}
  Let $ A $ be a Johnson pseudo-contractible Banach algebra and $ I $ be a two sided closed ideal of $ A $. If $ I $ has a bounded approximate identity, then $ I $ is Johnson pseudo-contractible.
\end{lem}
\begin{proof}
  By hypothesis there is a net $(M_{\alpha})\subseteq (A\otimes_{p}A)^{**}$ such that $a\cdot M_{\alpha}=M_{\alpha}\cdot a$ and $\pi^{**}_{A}(M_{\alpha})a-a\rightarrow 0$ for every $a\in A$. Let $ (e_{\beta}) $ be a bounded approximate identity for $ I $ and $  E $ be a weak* cluster point of $ (e_{\beta}) $ in $ I^{\ast\ast} $. Then by  setting $ (N_{\alpha})=(E\cdot M_{\alpha}\cdot E)\subseteq (I\otimes_{p}I)^{**} $, we have
  $$ x\cdot N_{\alpha}=N_{\alpha}\cdot x ,  $$
  and
  $$ \pi^{**}_{I}(N_{\alpha})x=\pi^{**}_{A}(E\cdot M_{\alpha}\cdot E)x=\pi^{**}_{A}(M_{\alpha})x\longrightarrow x, $$
  for any $ x\in I $. It follows that $ I $ is Johnson pseudo-contractible.
  \end{proof}
\begin{thm}\label{joh of A*B}
    Let $ A $ and $ B $ be two Banach algebras with $ \theta\in\Delta(B) $. If $ A\times_{\theta}B $ is Johnson pseudo-contractible, then the following statements hold
  \begin{enumerate}
    \item $ A $ is Johnson pseudo-contractible and has a bounded approximate identity,
    \item $ B $ is Johnson pseudo-contractible.
  \end{enumerate}
\end{thm}
\begin{proof}
 Suppose that $ \Phi:(A\times_{\theta} B)\otimes_{p}(A\times_{\theta} B)\longrightarrow A\times_{\theta} B $ is a linear map determined by
  \begin{equation*}
    \Phi((a, b)\otimes(c, d))=\theta(d)(a, b), \qquad (a,c\in A,\text{and}\quad b,d\in B).
  \end{equation*}
  Let $ (U_{\alpha})\subseteq ((A\times_{\theta} B)\otimes_{p}(A\times_{\theta} B))^{\ast\ast} $ be such that
  \begin{equation*}
   (a, b)\cdot U_{\alpha}=U_{\alpha}\cdot (a, b),\quad \pi_{A\times_{\theta} B}^{\ast\ast}(U_{\alpha})(a, b)\rightarrow (a, b),
  \end{equation*}
  for all $ a\in A $ and $ b\in B $. Then by Goldstein theorem for each $ \alpha $ there is a net $ (u_{\alpha_{\beta}})$ in $ (A\times_{\theta} B)\otimes_{p}(A\times_{\theta} B) $ such that $ w^{\ast}-\lim\limits_{\beta} u_{\alpha_{\beta}}=U_{\alpha} $. Suppose that $ u_{\alpha_{\beta}}=\sum\limits_{i=1}^{\infty}(a_{i}^{\alpha_{\beta}}, b_{i}^{\alpha_{\beta}})\otimes(c_{i}^{\alpha_{\beta}}, d_{i}^{\alpha_{\beta}}) $ for some sequences $ (a_{i}^{\alpha_{\beta}}), (c_{i}^{\alpha_{\beta}})\subseteq A $ and $ (b_{i}^{\alpha_{\beta}}), (d_{i}^{\alpha_{\beta}})\subseteq B $. Note that $ \theta $ has an extension $ \tilde{\theta}\in\Delta(B^{\ast\ast}) $ given by $ \tilde{\theta}(F)=F(\theta) $ for any $ F\in B^{\ast\ast} $. Since $ \Phi $ and $ \theta $ are bounded, $ \Phi^{\ast\ast} $ and $ \tilde{\theta} $ are weak* continuous maps. Now we have
  \begin{equation*}
    \begin{split}
      \langle(0, \tilde{\theta}), \Phi^{\ast\ast}(U_{\alpha}) \rangle
         &=w^{\ast}-\lim\limits_{\beta}\langle(0, \theta), \Phi(u_{\alpha_{\beta}}) \rangle   \\
         &=w^{\ast}-\lim\limits_{\beta} \sum\limits_{i=1}^{\infty}\theta(b_{i}^{\alpha_{\beta}})\theta(b_{i}^{\alpha_{\beta}})\\
         &=w^{\ast}-\lim\limits_{\beta}\langle(0, \theta), \pi_{A\times_{\theta} B}(u_{\alpha_{\beta}}) \rangle   \\
         &=\langle(0, \tilde{\theta}), \pi_{A\times_{\theta} B}^{\ast\ast}(U_{\alpha}) \rangle\longrightarrow 1.
    \end{split}
  \end{equation*}
  Set $ \Phi^{\ast\ast}(U_{\alpha}) =(\phi_{\alpha}, \psi_{\alpha}) $, where $ \phi\in A^{\ast\ast} $, $ \psi\in B^{\ast\ast} $. It is readily seen that $ \tilde{\theta}(\psi_{\alpha})\rightarrow 1 $. Take $ \alpha_{0} $ such that $ \tilde{\theta}(\psi_{\alpha_{0}})\neq 0 $, for each $ a\in A $ we have
  \begin{equation*}
    a\Phi^{\ast\ast}(U_{\alpha_{0}})=\Phi^{\ast\ast}(a\cdot U_{\alpha_{0}})=\Phi^{\ast\ast}(a\cdot U_{\alpha_{0}})-\Phi^{\ast\ast}( U_{\alpha_{0}}\cdot a)=\Phi^{\ast\ast}(a\cdot U_{\alpha_{0}}-U_{\alpha_{0}}\cdot a)=0.
  \end{equation*}
  Also, we have
  \begin{equation*}
    a\Phi^{\ast\ast}(U_{\alpha_{0}})=(a, 0)(\phi_{\alpha_{0}}, \psi_{\alpha_{0}})=(a\phi_{\alpha_{0}}+\tilde{\theta}(\psi_{\alpha_{0}})a, 0).
  \end{equation*}
  Therefore $ a\phi_{\alpha_{0}}+\tilde{\theta}(\psi_{\alpha_{0}})a=0 $, so $ a(-\tilde{\theta}(\psi_{\alpha_{0}})^{-1}\phi_{\alpha_{0}})=a $, where $-\tilde{\theta}(\psi_{\alpha_{0}})^{-1}\phi_{\alpha_{0}}\in A^{**}$. This show that $ A^{**} $ has a right  identity. An analogous argument show that $ A^{**} $ has a left  identity. It follows that $ A $ has a bounded approximate identity.  Since $ A $ is a two sided closed ideal of $ (A\times_{\theta}B) $ and has a bounded approximate identity is Johnson pseudo-contractible by Lemma \ref{lem  I  is Johnson pseudo-contractible}.

 It is well known that $ (A\times_{\theta}B)/A\cong B $ and there is a surjective homomorphism from $ A\times_{\theta}B $ onto $ (A\times_{\theta}B)/A $. So \cite[Proposition 2.9]{sah1} implies Johnson pseudo-contractibility of $ B $.
\end{proof}
We remark that the converse of recent theorem does not hold in general. For example $ A(H) $, the Fourier algebra on the integer Heisenberg group $ H $, is Johnson pseudo-contractible and has a bounded approximate identity. Also $ M(G) $ is Johnson pseudo-contractible for any discrete and amenable group $ G $, specially for $H$. But $ A(H)\times_{\theta}M(H) $ is not Johnson pseudo-contractible. Indeed, $ A(H)\times_{\theta}M(H) $ has an identity by \cite[Proposition 2.3]{monfared}. If $ A(H)\times_{\theta}M(H) $ is Johnson pseudo-contractible, then by Theorem \ref{joh imply amenable} $ A(H)\times_{\theta}M(H) $ must be amenable and \cite[page 285]{monfared} implies amenability of $ A(H) $. This lead to a contradiction that $ H $ has an abelian subgroup of finite index, see \cite[Theorem 2.3]{forrest}.

From \cite[page 285]{monfared} and Theorem \ref{joh imply amenable} we have the following Corollary.
\begin{cor}
  If $ B $ has an identity, then the following statements are equivalent
  \begin{enumerate}
    \item $ A\times_{\theta} B $ is Johnson pseudo-contractible,
    \item $ A\times_{\theta} B $ is amenable,
    \item $ A$ and $ B $ are amenable.
  \end{enumerate}
\end{cor}
\begin{cor}\label{cor A has an identity}
  If $ A $ has an identity, then $ A\times_{\theta} B $ is Johnson pseudo-contractible if and only if $ A $ is amenable and $ B $ is Johnson pseudo-contractible.
\end{cor}
\begin{proof}
  In view of \cite{choi} $ A\times_{\theta} B $ is nothing but the $\ell^{1}$-direct sum $ A\oplus B $  with coordinatewise product whenever $ A $ has an identity. If $ A $ is amenable and $ B $ is Johnson pseudo-contractible, then $ A\oplus B $ is Johnson pseudo-contractible by \cite[Theorem 2.11]{sah1}. The converse comes immediately from Theorem \ref{joh of A*B} and Theorem \ref{joh imply amenable}.
\end{proof}
It is well known that any dual Banach algebra with a bounded approximate identity has an identity, so we have the following corollary from Theorem \ref{joh of A*B} and Corollary \ref{cor A has an identity}.
\begin{cor}
  Let $ B $ be a Banach algebra and $ A $ be a dual Banach algebra with $ \theta\in\Delta(B) $. Then $ A\times_{\theta} B $ is Johnson pseudo-contractible if and only if $ A $ is amenable and $ B $ is Johnson pseudo-contractible.
\end{cor}
A Banach algebra $ A$ is called Arens regular if the first Arens product and the second Arens product on $ A^{\ast\ast} $ coincide, and is called weakly sequentially complete if every weakly Cauchy sequence is weakly convergent.
\begin{prop}\label{A is arens}
 Suppose that $ A$ and $ B $ are two Banach algebras with $ \theta\in\Delta(B) $. If $ A$ is Arens regular and weakly sequentially complete, then $ A\times_{\theta} B $ is Johnson pseudo-contractible if and only if
  \begin{enumerate}
    \item $ A $ is amenable and has an identity,
    \item $ B $ is Johnson pseudo-contractible.
  \end{enumerate}
\end{prop}
\begin{proof}
 If $ A\times_{\theta} B $ is Johnson pseudo-contractible, from Theorem \ref{joh imply amenable} $ A $ has a bounded approximate identity. By \"{U}lger theorem \cite[Theorem 2.9.39]{auto-dale}  $ A $ has an identity. Apply Corollary \ref{cor A has an identity}.
\end{proof}

It seems that Johnson pseudo-contractibility of $ A\times_{\theta} B $ is so related with amenability of $ A $.
We believe  that Corollary \ref{cor A has an identity} holds without the assumption that $ A $ has an identity. However, it remains a conjecture and needs proof, of course. We left it as an open problem in the following questions.
\begin{quest}
  Did Johnson pseudo-contractibility of $ A\times_{\theta} B $ implies amenability of $ A $?
\end{quest}
\begin{quest}
Suppose that $ A $ is an amenable Banach algebra and $B$ a Johnson pseudo-contractible Banach algebra with $ \theta\in\Delta(B) $. Is $ A\times_{\theta} B $ a Johnson pseudo-contractible Banach algebra?
\end{quest}
We finish this section with some examples, but first we recall that a semigroup $S$ is called regular if any $s\in S$ has an inverse, that is there exists
an element $t\in S$ such that $sts=s$. A semigroup $S$ is an inverse semigroup if any $s\in S$ has a unique inverse. The set of idempotents of semigroup $S$ is  denoted by $E(S)$. $E(S)$ becomes a partial ordered set with the following order
\begin{equation*}
  p\leq q \Longleftrightarrow p=pq= qp\qquad (p, q\in E(S)).
\end{equation*}
For $ p\in E(S) $, we set $ (p] = \{x : x \leq p\} $. An inverse semigroup $S$ is said to be uniformly locally finite if $ \sup\{\abs{(p]}: p\in E(S)\}<\infty $. It is well known that $\ell^{1}(S)$, the discrete semigroup algebra, is weakly sequentially complete \cite[Theorem A.4.4]{auto-dale}. Our refer for semigroup theory is \cite{dal-lau-str}.
\begin{exam}
Suppose that $ B $ is a Banach algebra and $ \theta\in\Delta(B) $.
  \begin{enumerate}
    \item [(i)]  Let $S$ be a uniformly locally finite inverse semigroup. Then Johnson pseudo-contractibility of $ \ell^{1}(S)\times_{\theta} B $ implies that $\ell^{1}(S)$ is Johnson pseudo-contractible and has a bounded approximate identity. From \cite[Proposition 2.1]{ram} $E(S)$ must be finite and from \cite[Theorem 2.3]{Sahami2017} each maximal subgroup of $S$ is amenable, in other word $\ell^{1}(S)$ is amenable, see \cite{ddun-pat}.
    \item [(ii)]  Suppose that $S$ is regular and  $\ell^{1}(S)$ is Arens regular. If $ \ell^{1}(S)\times_{\theta} B $ is Johnson pseudo-contractible then by Proposition \ref{A is arens} $\ell^{1}(S)$ is amenable and has an identity. So by \cite{ddun-pat} $E(S)$ is finite and that $S$ must be unital finite semigroup, see \cite[Theorem 12.2]{dal-lau-str}. Indeed, $ \ell^{1}(S)\times_{\theta} B $ is Johnson pseudo-contractible if and only if $S$ is unital finite semigroup and $ B $ is Johnson pseudo-contractible.
  \end{enumerate}
\end{exam}
\begin{exam}
 Using \cite[Theorem 3.1]{essl} one can see that $ M_{I}(\mathbb{C}) $ (the Banach algebra of $I\times
I $-matrices over $\mathbb{C}$, with finite $\ell^{1}$-norm and
matrix multiplication) has no bounded approximate identity unless $I$ is finite, but in this case $ M_{I}(\mathbb{C}) $ is amenable and has an identity. So for Banach algebra $ B $ and $ \theta\in\Delta(B) $, $ M_{I}(\mathbb{C})\times_{\theta} B $ is Johnson pseudo-contractible if and only if $I$ is finite and $ B $ is Johnson pseudo-contractible.
\end{exam}
A linear subspace $S^{1}(G)$ of $L^{1}(G)$ is said to be a Segal algebra on $G$ if it satisfies the following conditions
\begin{enumerate}
\item [(i)] $S^{1}(G)$ is  dense    in $L^{1}(G)$,
\item [(ii)]  $S^{1}(G)$ with a norm $\norm{\cdot}_{S^{1}(G)}$ is
a Banach space and $\norm{f}_{L^{1}(G)}\leq\norm{f}_{S^{1}(G)}$ for
every $f\in S^{1}(G)$,
\item [(iii)]  $ S^{1}(G) $ is left translation invariant (ie: $L_{y}f\in S^{1}(G)$ for every $f\in S^{1}(G)$ and $y\in G$) and the map $y\mapsto L_{y} (f)$ from $G$ into $S^{1}(G)$ is continuous, where
$L_{y}(f)(x)=f(y^{-1}x)$,
\item [(iv)] $\norm{L_{y}(f)}_{S^{1}(G)}=\norm{f}_{S^{1}(G)}$ for every $f\in
S^{1}(G)$ and $y\in G$.
\end{enumerate}
\begin{exam}
  Let $ B $ be a Banach algebra and $ \theta\in\Delta(B) $. Then for any proper Segal algebra $ S^{1}(G) $, $ S^{1}(G)\times_{\theta} B $ is not Johnson pseudo-contractible, since $ S^{1}(G) $ has no bounded approximate identity.
\end{exam}

%%%%%%%%%%%%%%%%%%%%%%%%%%%%%%%%%%%%%%%%%%%%%%%%%%%%%%%%%%%%%%%%%%%%

\section{pseudo-amenability of $ A\times_{\theta} B $}\label{sec-pse}

\begin{rem}\label{rem}
  Note that if $ U\in (A\times_{\theta}B)\otimes_{p}(A\times_{\theta}B) $, then there is $ M\in A\otimes_{p}A $, $ N\in A\otimes_{p}B $, $ L\in B\otimes_{p}A $ and $ H\in B\otimes_{p}B $ such that
  \begin{equation*}
    U=M+N+L+H,
  \end{equation*}
  and
  \begin{equation*}
    \norm{U}_{(A\times_{\theta}B)\otimes_{p}(A\times_{\theta}B)}=\norm{M}_{A\otimes_{p}A }+\norm{N}_{A\otimes_{p}B }+\norm{L}_{B\otimes_{p}A}+\norm{H}_{B\otimes_{p}B}.
  \end{equation*}
  \end{rem}

\begin{thm}\label{pse of A*B}
  Let $ A $ and $ B $ be two Banach algebras with $ \theta\in\Delta(B) $. If $ A\times_{\theta}B $ is pseudo-amenable, then the following statements hold
  \begin{enumerate}
    \item $ A $ is approximate amenable, and
    \item $ B $ is pseudo-amenable.
  \end{enumerate}
\end{thm}
\begin{proof}
  It is well known that $ (A\times_{\theta}B)/A\cong B $ and there is a surjective homomorphism from $ A\times_{\theta}B $ onto $ (A\times_{\theta}B)/A $. So \cite[Proposition 2.2]{ghah-pse} implies pseudo-amenability of $ B $.

  By assumption there is a net $ (U_{\alpha})\subseteq (A\times_{\theta}B)\otimes_{p}(A\times_{\theta}B) $  such that
  \begin{equation*}
   (x, y)\cdot U_{\alpha}-U_{\alpha}\cdot(x, y)\rightarrow 0, \quad \pi(U_{\alpha})(x, y)\rightarrow (x, y),
  \end{equation*}
  for each $ x\in A $, $ y\in B $. Particulary for each $ x\in A $ we have
  \begin{equation}\label{eq0}
  (x, 0)\cdot U_{\alpha}-U_{\alpha}\cdot (x, 0)\rightarrow 0, \quad \pi(U_{\alpha})(x, 0)\rightarrow (x, 0).
  \end{equation}
  Suppose that $ U_{\alpha}=\sum\limits_{i=1}^{\infty}(a_{i}^{\alpha}, b_{i}^{\alpha})\otimes(c_{i}^{\alpha}, d_{i}^{\alpha}) $ for some sequences $ (a_{i}^{\alpha}), (c_{i}^{\alpha})\subseteq A $ and $ (b_{i}^{\alpha}), (d_{i}^{\alpha})\subseteq B $. Set $ M_{\alpha}=\sum\limits_{i=1}^{\infty}a_{i}^{\alpha}\otimes c_{i}^{\alpha} $, $ F_{\alpha}=-\sum\limits_{i=1}^{\infty}\theta(d_{i}^{\alpha})a_{i}^{\alpha} $, $ G_{\alpha}=-\sum\limits_{i=1}^{\infty}\theta(b_{i}^{\alpha})c_{i}^{\alpha} $ and $ H_{\alpha}=\sum\limits_{i=1}^{\infty}b_{i}^{\alpha}\otimes d_{i}^{\alpha} $. Easily one can see that
  \begin{equation*}
    \pi_{A\times_{\theta}B}(U_{\alpha})=(\pi_{A}(M_{\alpha})-F_{\alpha}-G_{\alpha}, \pi_{B}(H_{\alpha})).
  \end{equation*}
  For an arbitrary element $ b $ of $ B $ we have
  \begin{equation*}
    \pi_{A\times_{\theta}B}(U_{\alpha})(0, b)=(\theta(b)(\pi_{A}(M_{\alpha})-F_{\alpha}-G_{\alpha}), \pi_{B}(H_{\alpha})b)\longrightarrow(0, b),
  \end{equation*}
  so
  \begin{equation*}
    \pi_{A}(M_{\alpha})-F_{\alpha}-G_{\alpha}\rightarrow 0, \qquad \theta(\pi_{B}(H_{\alpha}))\rightarrow 1.
  \end{equation*}
Note that
  \begin{equation}\label{eq1}
  \begin{split}
   (x, 0)\cdot U_{\alpha}
   &=\sum\limits_{i=1}^{\infty}(x, 0)(a_{i}^{\alpha}, 0)\otimes(c_{i}^{\alpha}, 0)+\sum\limits_{i=1}^{\infty}(x, 0)(0, b_{i}^{\alpha})\otimes(c_{i}^{\alpha}, 0)\\
   &+\sum\limits_{i=1}^{\infty}(x, 0)(a_{i}^{\alpha}, 0)\otimes(0, d_{i}^{\alpha})+\sum\limits_{i=1}^{\infty}(x, 0)(0, b_{i}^{\alpha})\otimes(0, d_{i}^{\alpha})\\
   &=x\cdot(\sum\limits_{i=1}^{\infty}(a_{i}^{\alpha}\otimes c_{i}^{\alpha})) +\sum\limits_{i=1}^{\infty}(x\otimes\theta(b_{i}^{\alpha})c_{i}^{\alpha}) +\sum\limits_{i=1}^{\infty}(xa_{i}^{\alpha}\otimes d_{i}^{\alpha}) +\sum\limits_{i=1}^{\infty}(\theta(b_{i}^{\alpha})x\otimes d_{i}^{\alpha})\\
   &=x\cdot M_{\alpha}-x\otimes G_{\alpha}+\sum\limits_{i=1}^{\infty}(xa_{i}^{\alpha}\otimes d_{i}^{\alpha})+\sum\limits_{i=1}^{\infty}(\theta(b_{i}^{\alpha})x\otimes d_{i}^{\alpha}),
   \end{split}
  \end{equation}
  similarly we have
  \begin{equation}\label{eq2}
    U_{\alpha}\cdot (x, 0)
   =M_{\alpha}\cdot x -F_{\alpha}\otimes x+\sum\limits_{i=1}^{\infty}(b_{i}^{\alpha}\otimes c_{i}^{\alpha}x)+ \sum\limits_{i=1}^{\infty}(b_{i}^{\alpha}\otimes\theta(d_{i}^{\alpha})x).
  \end{equation}
 From \eqref{eq1}, \eqref{eq2} and \eqref{eq0}, by using Remark \ref{rem} we obtain that
  \begin{enumerate}
    \item $x\cdot M_{\alpha}-M_{\alpha}\cdot x+F_{\alpha}\otimes x-x\otimes G_{\alpha}\rightarrow 0$,
    \vspace{.2cm}
    \item $\sum\limits_{i=1}^{\infty}(xa_{i}^{\alpha}\otimes d_{i}^{\alpha})+ \sum\limits_{i=1}^{\infty}(\theta(b_{i}^{\alpha})x\otimes d_{i}^{\alpha})\rightarrow 0$,
    \item $\sum\limits_{i=1}^{\infty}(b_{i}^{\alpha}\otimes c_{i}^{\alpha}x)+ \sum\limits_{i=1}^{\infty}(b_{i}^{\alpha}\otimes\theta(d_{i}^{\alpha})x)\rightarrow 0$.
  \end{enumerate}
  Define a bounded linear map $ \phi:A\otimes_{p}B\rightarrow A $ by $\phi(a\otimes b)=\theta(b)a$.
  From (2) we have
  \begin{equation*}
    -xF_{\alpha}+\theta(\pi_{B}(H_{\alpha}))x=x\sum\limits_{i=1}^{\infty}\theta(d_{i}^{\alpha})a_{i}^{\alpha}+ \sum\limits_{i=1}^{\infty}\theta(b_{i}^{\alpha}d_{i}^{\alpha})x =\phi(\sum\limits_{i=1}^{\infty}(xa_{i}^{\alpha}\otimes d_{i}^{\alpha})+ \sum\limits_{i=1}^{\infty}(\theta(b_{i}^{\alpha})x\otimes d_{i}^{\alpha}))\rightarrow 0,
  \end{equation*}
  now $ \theta(\pi_{B}(H_{\alpha}))\rightarrow 1 $ implies that $ xF_{\alpha}\rightarrow x $. Similarly by using (3) we have  $ G_{\alpha}x\rightarrow x $. So we find $ (M_{\alpha})\subseteq A\otimes_{p}A $, $ (F_{\alpha})\subseteq A $ and $ (G_{\alpha})\subseteq A $ such that
  \begin{enumerate}
    \item $x\cdot M_{\alpha}-M_{\alpha}\cdot x+F_{\alpha}\otimes x-x\otimes G_{\alpha}\rightarrow 0$,
    \item $xF_{\alpha}\rightarrow x$, \quad $ G_{\alpha}x\rightarrow x  $,
    \item $ \pi_{A}(M_{\alpha})x-F_{\alpha}x-G_{\alpha}x\rightarrow 0 $,
  \end{enumerate}
  for every $ x\in A $. It follows that $ A $ is approximately amenable.
\end{proof}
\begin{exam}
  Let $S$ be a uniformly locally finite inverse semigroup and $ B $ be a Banach algebra with $ \theta\in\Delta(B) $. If $ \ell^{1}(S)\times_{\theta}B $ is pseudo-amenable, then by Theorem \ref{pse of A*B} $ \ell^{1}(S) $ is approximately amenable.  Theorem 4.3 of \cite{ros2013} show that $ \ell^{1}(S) $ is amenable.
\end{exam}
\begin{exam}
Let $ G=SU(2) $ be  the $2 \times 2$ unitary group, and suppose that $ S^{1}(G) $  is any proper Segal algebra on $ G $. We claim that $ S^{1}(G)\times_{\theta}S^{1}(G) $ is not pseudo-amenable. To see this we go towards a contradiction and suppose that $ S^{1}(G)\times_{\theta}S^{1}(G) $ is pseudo-amenable. By Theorem \ref{pse of A*B} $ S^{1}(G) $ is approximately amenable, which is a contradiction with the main result of \cite{alagh-app}.
\end{exam}
\begin{exam}
  Let $G$ be an infinite abelian compact group and $ B $ be a Banach algebra with $ \theta\in\Delta(B) $. We claim that $ L^{2}(G)\times_{\theta}B $ is not pseudo-amenable. To see this suppose that $ L^{2}(G)\times_{\theta}B $ is pseudo-amenable. Then Theorem \ref{pse of A*B} implies that $ L^{2}(G)$ is approximately amenable. But by Plancherel theorem $ L^{2}(G)$ is isometrically isomorphism to $ \ell^{2}(\hat{G})$, where $ \hat{G} $ is denoted the dual group of $G$. So $ \hat{G} $ is approximately amenable which is a contradiction with the main result of \cite{dal-loy-zhang}.
\end{exam}

% ----------------------------------------------------------------

\end{document}